\documentclass[12pt,reqno,twoside,letterpaper]{amsart}

\usepackage{hyperref}

\makeatletter
\renewenvironment{proof}[1][\proofname]{\par
  \pushQED{\qed}
  \normalfont \partopsep=\z@skip \topsep=\z@skip
  \trivlist
  \item[\hskip\labelsep
    \itshape
  #1\@addpunct{.}]\ignorespaces
}{
  \popQED\endtrivlist\@endpefalse\vspace{0.3em}
}
\makeatother

\newcommand{\M}{\mathbb{M}}
\newcommand{\N}{\mathbb{N}}
\renewcommand{\P}{\mathbb{P}}
\newcommand{\T}{\mathbb{T}}

\newtheorem{proposition}{Proposition}
\newtheorem{question}{Question}

\begin{document}

\title[A note on OEIS sequence A111384]{A note on OEIS sequence A111384}
\author{Markus Sigg}
\address{Freiburg, Germany}
\email{mail@markussigg.de}
\date{September 26, 2023.}

\begin{abstract}
  $A111384(n)$ is an upper bound for the number of primes that can be written
  as a sum of three distinct primes selected from a set of $n$ primes. Is this
  bound sharp?
\end{abstract}

\maketitle

{
  AMS subjects classification 2010: 05A20,11B83,11L20.
}

\section{Introduction}

With the natural numbers $\N := \{0,1,2,\dots\}$, let $\T := \{ 3 \} \cup (\N
\setminus 3\N)$. For non-empty $\M \subset \T$ and finite $A \subset \M$ set
\[
  S_\M(A) \ := \ \M \cap \{ a + b + c : a,b,c \in A, \ a < b < c \}
\]
and for $n \in \N$:
\[
  s_\M(n) \ := \ \max \, \{ |S_\M(A)| : A \subset \M, \ |A| = n \}
\]

We shall prove that OEIS sequence A111384, see \cite{A111384}, gives an upper
bound for $s_\M(n)$, i.\,e.
\[
  s_\M(n) \ \le \ A111384(n) \ = \
  \binom{n}{3} -
  \binom{\lfloor \frac n2 \rfloor}{3} -
  \binom{\lceil  \frac n2 \rceil}{3} ,
\]
which therefore is true in particular for $\M = \P$, the set of prime numbers.
Furthermore we will show that this inequality is in fact an equality in the
case of $\M = \T$, and dare to ask if it is an equality even in the case of
$\M = \P$.

Let's remark beforehand that A111384 starts A111384(0) = A111384(1) = A111384(2)
= 0, and in general
\begin{equation} \label{eq1}
  A111384(n) = \begin{cases}
    \displaystyle \frac18 (n-2)n^2     & \text{for even $n$},\\[1em]
    \displaystyle \frac18 (n-2)(n^2-1) & \text{for odd $n$},
  \end{cases}
\end{equation}
thus for all $n \in \N$:
\begin{equation} \label{eq2}
  A111384(n) \ \ge \ \frac18 (n-2)(n^2-1)
\end{equation}

\section{Statements}

\begin{proposition} \label{prop}
  $s_\M(n) \le A111384(n)$ for all $\M \subset \T$ and $n \in \N$.
\end{proposition}

\begin{proof}
  This is trivial for $n < 3$, so let $n \ge 3$ and $A \subset \M$ with
  $|A| = n$. For $m \in \{0,1,2\}$ set
  $A_m := \{ a \in A : a \equiv m \pmod 3 \}$ and $\alpha_m := |A_m|$.
  We have to show $|S_\M(A)| \le A111384(n)$.

  The case of $3 \not\in A$: Here, $A = A_1 \uplus A_2$. There are
  $t(n) := n(n-1)(n-2)/6$ triples $(a,b,c)$ with $a,b,c \in A$ and $a < b < c$,
  see \cite{A000292}. Because $3 \ | \ a + b + c$ for $a,b,c \in A_1$ or
  $a,b,c \in A_2$, at least
  $t(\alpha_1) + t(\alpha_2) = t(\alpha_1) + t(n-\alpha_1)$ of these triples
  cannot contribute to $S_\M(A)$, hence
  \[
    |S_\M(A)|
    \ \le \ t(n) - t(\alpha_1) - t(n-\alpha_1)
    \ = \ f_n(\alpha_1),
  \]
  where for $\alpha \in \{0, \dots, n\}$
  \[
    f_n(\alpha) \ := \ \frac18 (n-2)n^2 - \frac12 (n - 2) \left(\alpha - \frac n2\right)^2.
  \]
  For even $n$, the maximum of $f_n$ is $f_n(n/2) = (n-2)n^2/8$. For odd $n$,
  the maximum of $f_n$ is $f_n((n-1)/2) = (n-2)(n^2-1)/8$. With (\ref{eq1}) we
  get $|S_\M(A)| \le f_n(\alpha_1) \le A111384(n)$.

  The case of $3 \in A$: Here, $A = \{3\} \uplus A_1 \uplus A_2$. For
  $a,b \in A_1 \cup A_2$, $3 + a + b \in \M$ is possible only when
  $a,b \in A_1$ or $a,b \in A_2$, so
  \begin{eqnarray*}
    |S_\M(A)|
    &\le& |S_\M(A_1 \cup A_2)| + \frac12 \alpha_1(\alpha_1-1) + \frac12 \alpha_2(\alpha_2-1)\\
    &\le& f_{n-1}(\alpha_1) + \frac12 \alpha_1(\alpha_1-1) + \frac12 (n-1-\alpha_1)(n-2-\alpha_1)\\
    &=& g_n(\alpha_1),
  \end{eqnarray*}
  where for $\alpha \in [0, n-1]$
  \[
    g_n(\alpha) \ := \ \frac12 (n-1)(n-2) + \frac12 (n-5)(n-1-\alpha)\alpha.
  \]
  The maximum of $g_n$ is $g_n((n-1)/2) = (n^3 - 3n^2 - n + 3)/8 =: v(n)$.
  With inequality (\ref{eq2}) we get $A111384(n) - v(n) \ge (n^2-1)/8 > 0$, so
  $|S_\M(A)| \le g_n(\alpha_1) \le v(n) < A111384(n)$.
\end{proof}

\begin{proposition}
    $s_\T(n) = A111384(n)$ for all $n \in \N$.
\end{proposition}

\begin{proof}
  For $k \in \N$ set $a_k := 3^{k+1}+1$ for odd $k$ and $a_k := 3^{k+1}+2$ for
  even $k$, and with this $A_n := \{a_1,\dots,a_n\} \subset \T$ for $n \in \N$.

  Looking at its representation in base $3$ makes it obvious that $a+b+c$ for
  $a,b,c \in A_n, a<b<c$ is unique, i.\,e. $a+b+c \ne x+y+z$ for
  $a,b,c,x,y,z \in A_n$ with $a < b < c, x < y < z$ and $(a,b,c) \ne (x,y,z)$.
  The considerations in the first part of the proof of Proposition \ref{prop}
  show $|S_\T(A_n)| = A111384(n)$.
\end{proof}

Proposition \ref{prop} shows in particular that $s_\P(n) \le A111384(n)$ for
all $n \in \N$. From \cite{vps} it is known that $s_\P(n) = A111384(n)$ for
$n \le 9$, for example
\[
  |S_\P(\{499,1483,2777,4363,5237,5507,6043,6197\})| = 48 = A111384(8).
\]
This leads to
\begin{question}
   $s_\P(n) = A111384(n)$ for all $n \in \N$?
\end{question}

\end{document}